\documentclass{amsart}
\usepackage[utf8]{inputenc}
\usepackage[hidelinks]{hyperref}
\usepackage{url}

\usepackage{nomencl}
\makenomenclature

\usepackage{mathtools}
\usepackage{enumerate}
\usepackage{cite}
\usepackage[shortlabels]{enumitem}
\usepackage{tikz-cd}
\usepackage{stmaryrd}

\usepackage{latexsym}
\usepackage{amsmath,amssymb,amsthm,amscd}
\usepackage{color}
\usepackage[nolabel, final]{showlabels}

\newtheorem{thm}{Theorem}[section]

\newtheorem{lem}[thm]{Lemma}
\newtheorem{cor}[thm]{Corollary}

\theoremstyle{definition}
\newtheorem{rem}[thm]{Remark}
\newtheorem{defn}[thm]{Definition}

\numberwithin{equation}{section}

\newcommand{\NN}{\mathbb{N}}

\newcommand{\Hom}{\operatorname{Hom}}
\newcommand{\Rhom}{\mathit{R}\Hom}

\newcommand{\ZZ}{\mathbb{Z}}

\newcommand{\QQ}{\mathbb{Q}}
\newcommand{\Q}{\mathbb{Q}}

\newcommand{\CC}{\mathbb{C}}

\newcommand{\fB}{\mathbf{B}}

\newcommand{\cR}{\mathcal{R}}

\newcommand{\cone}{\operatorname{cone}}
\newcommand{\id}{\operatorname{id}}

\DeclarePairedDelimiter\abs{\lvert}{\rvert}%

\begin{document}

\title{Comparisons of Lie algebra cohomologies of $(\varphi,\Gamma)$-modules}


\author{Rustam Steingart}
\address{Rustam Steingart\\
	Ruprecht-Karls-Universität Heidelberg,
	Mathematisches Institut \\
	Im Neuenheimer Feld 205 \\
	69120 Heidelberg\\
	Germany}
\email{rsteingart@mathi.uni-heidelberg.de}

\subjclass[2020]{11F80, 11S25, 22E35}
\keywords{$(\varphi,\Gamma)$-modules, Galois cohomology, analytic cohomology}

\thanks{
	I would like to thank Muhammad Manji for pointing out the spectral sequences considered in section 3 to me and discussing related questions, which ultimately motivated me to write this article. I would also like to thank Gautier Ponsinet and Otmar Venjakob for valuable discussions and their comments on earlier drafts of this article. 
	This research was funded by the Deutsche Forschungsgemeinschaft (DFG, German Research Foundation) TRR 326 \textit{Geometry and Arithmetic of Uniformized Structures}, project number 444845124.}

\maketitle

\begin{abstract}
We generalise a result of Fourquaux and Xie thereby completely determining the relationship between $\QQ_p$-analytic and $L$-analytic Lie algebra cohomology of analytic $(\varphi_L,\Gamma_L)$-modules. We use the results to conclude that for $L\neq \QQ_p,$ there exist examples of étale $(\varphi_L,\Gamma_L)$-modules over Robba rings whose $\QQ_p$-analytic cohomology does not arise as a base change of Galois cohomology. 
\end{abstract}

\bigskip
	\section{Introduction}
Let $L/\mathbb{Q}_p$ be a finite extension of degree $d$ and $\varphi_L \in o_L \llbracket T\rrbracket$ a Frobenius power series for a uniformiser $\pi_L$ of $L.$ We denote by $L_\infty$ the Lubin-Tate extension attached to $\varphi_L$ and let $\Gamma_L:= \operatorname{Gal}(L_\infty/L).$ Contrary to the case $L=\QQ_p$ there exist Galois representations, which are not overconvergent and hence the categories of étale $(\varphi_L,\Gamma_L)$-modules over the Robba ring $\cR_L$ with coefficients in $L$ and the category of continuous $L$-linear representations are not equivalent. By a theorem of Berger there is an equivalence between $L$-analytic representations which are representations $V$ with the property that $\CC_p \otimes_{L,\sigma}V$ is the trivial $\CC_p$-semilinear representation for every $\sigma \neq \id$ and $L$-analytic étale $(\varphi_L,\Gamma_L)$-modules, i.e., étale $(\varphi_L,\Gamma_L)$-modules $M$ over $\cR_L$ such that the derived action $\operatorname{Lie}(\Gamma_L) \times M \to M $ is $L$-bilinear (cf.\cite{Berger2016}). The Lie algebra of $\Gamma_L$ is one dimensional over $L$ but $d$-dimensional over $\QQ_p.$ The $L$-bilinearity condition leads to a system of differential equations on $M$ which can be reinterpreted as $M$ being a torsion module over the universal enveloping algebra $U_{\QQ_p}(\mathfrak{g}_0)$ of the restriction of scalars $\mathfrak{g}_0$ of the Lie algebra $\mathfrak{g}$ of $\Gamma_L\footnote{For technical reasons we require a base change to a Galois closure $E$ of $L$ but we gloss over this in the introduction.}.$ This torsion relation will be used to relate Lie cohomology over $L$ and $\QQ_p,$ respectively. By definition the $L$-Lie algebra cohomology of $M$ is $\Rhom_{U(\mathfrak{g})}(L,M)$ and is related to analytic cohomology via a spectral sequence (cf. \cite{Kohlhaase}). 
In the context of $(\varphi_L,\Gamma_L)$-modules one usually studies cohomology theories of the form $$C_{f,?}(M):=\operatorname{cone}(C_?^\bullet(\Gamma_L,M) \xrightarrow{f-\id}C_?^\bullet(\Gamma_L,M))[-1],$$ where $f$ is an operator, which commutes with $\Gamma_L$ (usually $f = \varphi_L$ or the left-inverse $\Psi$) and $C^\bullet_?$ is a complex computing ``some'' $\Gamma_L$-cohomology with values in $M$ (e.g. $? \in \{\text{cts},\QQ_p\text{-an},L\text{-an},\QQ_p\text{-Lie},L\text{-Lie}$\}).
We will show the following theorem:
\begin{thm} Let $K$ be a field containing a Galois closure of $L,$
	let $M$ be an $L$-analytic $(\varphi_L,\Gamma_L)$-module over $\cR_K,$ and let $f \in \operatorname{End}_L(M)$ be an operator that commutes with the action of $\mathfrak{g}_0$ on $M.$ Then
	
	$$C^{\bullet}_{f,\mathbb{Q}_p\text{-Lie}}(M)  \simeq \bigoplus_{n=0}^{d-1} C^{\bullet}_{f,L\text{-Lie}}(M)[-n]^{\binom{d-1}{n}}.$$
	
\end{thm}

\begin{cor} 
	We have an isomorphism of $L$-vector spaces
	$$H^i_{cts}(M) \cong H^0_{an}(M)^{\binom{d-1}{i}} \oplus H^1_{an}(M)^{\binom{d-1}{i-1}}  \oplus H^2_{an}(M)^{\binom{d-1}{i-2}},$$
	for every $i \geq 0.$ In particular: \\
	$H^1_{cts}(M) \cong H^1_{an}(M) \oplus H^0_{an}(M)^{d-1}$ and $H^{d+1}_{cts}(M) \cong H^2_{an}(M).$
	
\end{cor}	
This completely determines the relationship between analytic and continuous cohomology of analytic $(\varphi_L,\Gamma_L)$-modules and generalises the results of Fourquaux-Xie, who determined the relationship of $H^0$ and $H^1$ in \cite{FX12}. Our result follows from a combinatorial argument and hence greatly simplifies their original proof. 
Furthermore, the above theorem can be generalised to $(\varphi_L,\Gamma_L)$-modules which have trivial ``Sen operator'' $\nabla_\sigma$ for a prescribed set of embeddings $\sigma \colon L \to \CC_p.$
In some cases the cohomology groups of $\Rhom_{L[\Gamma_L]}(L,C_{f,\mathfrak{g}}(M))$ are known to be finite-dimensional over the base field $K.$ In these cases we can relate the Euler-characterstics of the $\mathfrak{g}$ and $\mathfrak{g}_0$ variant to conclude that the Euler characteristic of the $\mathfrak{g}_0$-variant is zero if $d>1.$ This result is slightly surprising considering the relationship between cohomology of $(\varphi_L,\Gamma_L)$-modules and Galois cohomology. 
On the one hand, in the classical case $L=\QQ_p,$ the Galois cohomology of a representation is isomorphic to the continuous cohomology of its $(\varphi,\Gamma)$-module over the Robba ring (cf. \cite{liu2007cohomology}) and on the other hand,
working with a different coefficient ring, the Galois cohomology of an $L$-linear representation and the continuous cohomology of its Lubin-Tate $(\varphi_L,\Gamma_L)$-module are isomorphic (cf. \cite{kupferer2022herr}).   
Using an abstract formalism for Herr complexes allows us to define cup products in terms of Yoneda pairings.
If one considers an analytic $(\varphi_L,\Gamma_L)$-module which is generic, in the sense that $H^0_{an}(M)=H^2_{an}(M)=0,$ our results assert that $H^2_{cts}(M)\cong (H^1_{an}(M))^{d-1}.$ This is no coincidence and a source of non-trivial classes in $H^2_{cts}(M)$ is obtained by forming cup-products with non-analytic classes in $H^1_{cts}(L)$. If $H^1_{an}(M)$ is of rank one we even show that this is the only source of such classes.
\begin{thm}
	Let $M$ be an $L$-analytic $(\varphi_L,\Gamma_L)$-module with the property $H^0_{an}(M)=H^2_{an}(M)=0.$ 
	\begin{enumerate}
		\item The cup product defines a non-degenerate bilinear map $$H^1_{an}(M) \times H^1_{cts}(L)/H^1_{an}(L) \to H^2_{cts}(M).$$
		\item  If in addition $H^1_{an}(M)$ is one-dimensional over $L$, then the cup product induces an isomorphism $$H^1_{an}(M) \otimes_L H^1_{cts}(L)/H^1_{an}(L) \cong H^2_{cts}(M).$$
	\end{enumerate}
\end{thm}
In the last section we take a different perspective on the problem of comparing analytic and continuous cohomology. Instead of trying to compare the analytic cohomology of an $L$-analytic $(\varphi_L,\Gamma_L)$-module $M$ with the continuous cohomology of $M$ itself, we will compare the analytic cohomology to a suitable cohomology of a bigger object. For technical reasons we will be working with $(\varphi_L,\Gamma_L)$-modules over the analytic vectors $B:= (\tilde{\fB}^I_L)^{la}$ of certain period rings. 
Given a $(\varphi_L,\Gamma_L)$-module $N$ over $B$ we can find its $L$-analytic vectors $M:=\operatorname{Sol}(N)$ as solutions to an analogue of the Cauchy-Riemann differential equations (note that $M$ is a $(\varphi_L,\Gamma_L)$-module over the $L$-analytic vectors of $B$ rather than the ring $\cR_L$ considered before). Inspired by the Frölicher spectral sequence from complex analysis we prove an analogue in this case.
To this end we define a complex $C_{\Omega^{\Sigma_0}}^\bullet(N),$ analogous to the Dolbeaut complex, whose $0$-th cohomology is $M$ and another complex $C_{\Omega^{\Sigma}}^\bullet(N),$ which is obtained from the preceding by taking the $\nabla_{\id}$-cone.  If $N$ is attached to an $L$-analytic Galois representation we have the following result.

\begin{thm} 
	Let $V\in \operatorname{Rep}_{L}(G_L)$ be $L$-analytic, let $I=[r,s]\subset (0,1]$ be a closed subinterval with $r$ large enough. Set $$N^I:= (\tilde{D}^I(V))^{la}:= ((\tilde{\fB}^I\otimes V)^{H_L})^{la}$$ and $M^I:= \operatorname{Sol}(N^I).$
	Then the natural map $M^I[0] \to C_{\Omega^{\Sigma_0}}^\bullet(N^I)$ is an equivalence and $$H^i_{\Omega^{\Sigma}}(N^I) \cong H^i_{\nabla_1}(M^I)$$
	$$H^i_{\varphi_L,\Omega^{\Sigma}}(N^I) \cong H^i_{\varphi_L,\nabla_1}(M^I).$$
\end{thm}
	
\section{Preliminaries on $(\varphi_L,\Gamma_L)$-modules and cohomology}
\subsection{$(\varphi_L,\Gamma_L)$-modules}
Let $L/\QQ_p$ be a finite extension. We fix a uniformiser $\pi_L$ of $o_L$ and a Frobenius power series $\varphi_L$ for $\pi_L.$ We denote by $\Gamma_L:= \operatorname{Gal}(L_\infty/L)$ the Galois group of the corresponding Lubin-Tate extension. For a complete field extension $K\subset \CC_p$ of $L$ and a closed subinterval $I =[r,s] \subset [0,1]$ with $r,s \in \abs{\CC_p},$ we denote by $\cR_K^I$ the ring of Laurent series (power series if $r=0$) converging on the annulus $ r \leq \abs{T} \leq s.$ We endow these rings with their natural Banach topology which turns $\cR_K^{[r,s)}:= \varprojlim_{r<s'<s} \cR_K^{[r,s']}$ into a Fréchet space. The ring $\cR_K:= \varinjlim_{0<r<1} \cR_K^{[r,1)}$ is called the Robba ring with coefficients in $K$ and we endow it with the LF topology. The action of $\Gamma_L$ on $o_L\llbracket T \rrbracket$ extends to these rings and if $[r^{1/q},s^{1/q}]$ does not contain any zeroes of $\varphi_L,$ then we can extend the endomorphism $\varphi_L$ of $o_L\llbracket T\rrbracket$ to a map $\varphi_L \colon \cR_K^{[r,s]} \to \cR_K^{[r^{1/q},s^{1/q}]}.$ This is the case if $r > \abs{\pi_L}^{\frac{q}{q-1}}.$ In particular $\varphi_L$ extends to a $K$-linear endomorphism of $\cR_K.$ A $(\varphi_L,\Gamma_L)$-module over $\cR_K$ is a finite free $\cR_K$-module $M$ together with a continuous semi linear action of $\Gamma_L$ and a $\varphi_L$-semi linear endomorphism $\varphi_M$ which commutes with the action of $\Gamma_L$ such that the linearised map $$\varphi_M^{lin} \colon \cR_K \otimes_{\cR_K,\varphi_L}M \to M$$ is an isomorphism. 
One can show that such an $M$ arises as base change of a module $M_0$ defined over some half-open interval $[r_0,1)$ and that the matrix $A$ of $\varphi_{M_0} \colon M_0 \to \cR_K^{[r_0^{1/q},1)}\otimes_{\cR_K^{[r_0,1)}} M_0$ converges at the boundary $1$ and the valuation $v_p(\operatorname{det}(A))$ (at $T=1$) is called degree of $M$ (cf. \cite[Section 3.3]{Berger}). The above notion of degree together with the obvious notion of rank provides us with a theory of Harder-Narasimhan slopes for $(\varphi_L,\Gamma_L)$-modules. A $(\varphi_L,\Gamma_L)$-module is called étale if it is semi-stable of slope zero. The action of $\Gamma_L$ on $M$ induces an action of $\operatorname{Lie}(\Gamma_L)$ on $M.$ We say that $M$ is $L$-analytic if this action is $L$-bilinear. A continuous $L$-linear representation $V$ of the absolute Galois group $G_L$ of $L$ is called analytic if the $\CC_p$-semilinear representation $\CC_p \otimes_{L,\sigma}V$ is trivial for all $ \id \neq \sigma \in \Sigma:= \operatorname{Hom}_L(L,\CC_p).$ 
We recall that $o_L\llbracket T\rrbracket$ can be $(\varphi_L,\Gamma_L)$-equivariantly embedded into the ring of ramified Witt vectors $W(\widehat{L}_\infty^\flat)_L$ over the tilt $\widehat{L}_\infty^\flat$ of $\widehat{L}_\infty$ and one can define an analogue of the Robba ring on this level which we denote by $\tilde{\cR}_L$ (sometimes denoted $\tilde{\fB}^{\dagger}_{rig,L}$ in the literature). 

\begin{thm}
	\begin{enumerate}[(i)] We have the following equivalences of categories: 
		\item The category of étale $(\varphi_L,\Gamma_L)$-modules over $\tilde{\cR}_L$ is equivalent to the category of continuous $L$-linear representations of $G_L.$
		\item 	The category of étale $(\varphi_L,\Gamma_L)$-modules over $\cR_L$ is equivalent to the category of over-convergent $L$-linear representations of $G_L.$
		\item The category of étale $L$-analytic  $(\varphi_L,\Gamma_L)$-modules over $\cR_L$ is equivalent to the category of analytic $L$-linear representations of $G_L.$
	\end{enumerate}
	
\end{thm}
\begin{proof} The first equivalence is \cite[Proposition 6.3]{SV23comparing}.
	The second equivalence follows by combining \cite[Proposition 1.5 and Proposition 1.6]{FX12}. The third equivalence is \cite[Theorem D]{Berger2016}.
\end{proof}
\subsection{Cohomology}
Let $V$ be a complete LF space over a complete field $K \subset \CC_p$ with a pro $L$-analytic action of a compact $L$-analytic group $G$ which contains an open normal subgroup $U$ isomorphic to $o_L.$
By \cite[Proof of 4.3.10]{SchneiderVenjakobRegulator} the action of $U$ extends to a separately continous action of the algebra $D^{\QQ_p\text{-an}}(U,K)$ of $\QQ_p$-analytic distributions which factors over the algebra $D^{L\text{-an}}(U,K)$ of $L$-analytic distributions. By Amice' theorem $D^{\QQ_p\text{-an}}(U,K)$ is isomorphic to the ring of convergent power series on a $d$-dimensional polydisc. More explicitly, any $D^{\QQ_p\text{-an}}(U,K)$ can be expanded in a convergent power series in the variables $(\delta_{\gamma_1}-1,\dots, \delta_{\gamma_d}-1)$ where $\gamma_1,\dots,\gamma_d$ is a $\ZZ_p$-basis of $U$ and $\delta_\gamma$ denotes the dirac distribution corresponding to $\gamma \in U.$ Then the trivial $D^{\QQ_p\text{-an}}(U,K)$-module $F$ admits a projective resolution $P^\bullet$ given by the Koszul complex for $(\delta_{\gamma_1}-1,\dots, \delta_{\gamma_d}-1).$ 
Since $V$ being a complete locally convex space is complete and linearly topologised, the complex $\operatorname{Hom}_{D^{\QQ_p\text{-an}}(U,K)}(P^\bullet,V)$ computes continuous group cohomology of $U$ with values in $V$ by \cite[V,1.2.6]{lazard1965groupes}.
In  particular $$\operatorname{Ext}_{D^{\QQ_p\text{-an}}(U,K)}^i(K,V) \cong H^i_{cts}(U,V).$$ By variants of Shapiro's Lemma on both sides we also have
$$\operatorname{Ext}_{D^{\QQ_p\text{-an}}(G,K)}^i(K,V) \cong H^i_{cts}(G,V).$$
For the purpose of this article we define $L$-analytic cohomology as $$H^i_{an}(G,V):= \operatorname{Ext}^i_{D^{L\text{-an}}(G,K)}(K,V).$$  
We denote by $\mathfrak{g}$ (resp. $\mathfrak{g}_0$) the $L$-Lie algebra of $G$ (resp. the $\QQ_p$-Lie algebra of the restriction of scalars of $G$ to $\QQ_p$). Recall that the cohomology of a Lie algebra $\mathfrak{h}$ over a field $K$ with values in a $\mathfrak{h}$-module $V$ is given by the groups $\operatorname{Ext}^i_{U(\mathfrak{h})}(K,V).$ There are spectral sequences comparing $L$-Lie algebra cohomology with $L$-analytic cohomology (resp. $\QQ_p$-Lie algebra cohomology with continuous cohomology). One has the following 
\begin{thm} \label{thm:LieundAN}
	$$H^q_{an}(G,V) \cong H^q(\mathfrak{g},V)^G$$ and $$H^q_{cts}(G,V) \cong H^q(\mathfrak{g_0},V)^G.$$
\end{thm}
\begin{proof}
	Apply \cite[Theorem 4.10]{Kohlhaase}. Note that the theorem is applicable to our (solvable) group $G$ by \cite[Theorem 6.5]{Kohlhaase}. The assumption that $K$ is spherically complete is not required (cf. \cite[Section 4]{MSVW} for an elaboration).
\end{proof}

The situations above have a thing in common. In every case there is an augmented $K$-algebra $D$ such that the cohomology groups are given by $\operatorname{Ext}^i_D(K,M)$ where one takes $D = D^{\QQ_p\text{-an}}(G,K), D^{L\text{-an}}(G,K), U_K(\mathfrak{g})$ or $U_K(\mathfrak{g}_0).$
\begin{defn}
	Let $M$ be a $(\varphi_L,\Gamma_L)$-module over $\cR_K.$  We define 
	the continuous cohomology $C^\bullet_{cts}(M)$ as $$\operatorname{cone}(\Rhom_{D^{\QQ_p\text{-an}}(\Gamma_L,K)}(K,M) \xrightarrow{\varphi_L-1} \Rhom_{D^{\QQ_p\text{-an}}(\Gamma_L,K)}(K,M))[-1]$$ and
	the Lie-algebra cohomology $C^\bullet_{Lie}(M)$ as  $$\operatorname{cone}(\Rhom_{U_K(\operatorname{Lie}(\Gamma_L))_0)}(K,M) \xrightarrow{\varphi_L-1} \Rhom_{U_K((\operatorname{Lie}(\Gamma_L))_0)}(K,M))[-1].$$
	If $M$ is $L$-analytic we define in addition analytic cohomology 
	$C^\bullet_{an}(M)$ as $$\operatorname{cone}(\Rhom_{D^{L\text{-an}}(\Gamma_L,K)}(K,M) \xrightarrow{\varphi_L-1} \Rhom_{D^{L\text{-an}}(\Gamma_L,K)}(K,M))[-1]$$
	and $L$-Lie algebra cohomology $C^\bullet_{L-Lie}(M)$ as $$ \operatorname{cone}(\Rhom_{U_K(\operatorname{Lie}(\Gamma_L))}(K,M) \xrightarrow{\varphi_L-1} \Rhom_{U_K(\operatorname{Lie}(\Gamma_L))}(K,M))[-1].$$
\end{defn}

\section{Preliminary results on Koszul complexes}
We will need some combinatorics. 
\begin{defn} \label{def:combo}Define a sequence of finite sequences $(y_i)_{i \in \NN_0}$ by setting $y_0 = 0$ and $y_{i+1} = \widehat{y_{i}}*y_i$ for $i >0,$ where $*$ denotes concatenation of sequences and $\widehat{y_i}$ is the sequence obtained by increasing each entry of $y_i$ by $1.$ For $k \in \NN_0$ denote by $N(k,n)$ the number of occurrences of $k$ in the sequence $y_n.$
\end{defn} The sequence above starts with: $$(0), (1,0),(2,1,1,0), (3,2,2,1,2,1,1,0), \dots.$$
\begin{rem} \label{rem:combinatorics} The following hold:
	\begin{enumerate}[(i)]
		\item $N(0,i)=N(i,i) = 1$ for every $i \in \NN_0.$
		\item $N(1,i) = i$ for every $i \in \NN_0.$
		\item $N(k,n+1) = N(k-1,n) + N(k,n)$ for $k>0$ and every $n \in \NN_0.$
		\item $N(k,n)  = \binom{n}{k}.$
	\end{enumerate}
	\begin{proof}
		The points {\it{(i)}} and  {\it(iii)} follow from the definition and {\it{(ii)}} follows by combining {\it{(i)}} and {\it{(iii)}}.
		Point {\it{(iv)}} follows from from {\it{(i),(ii),(iii)}} by recurrence using the formula for binomial coefficients $$\binom{n}{k-1} + \binom{n}{k} = \binom{n+1}{k}.$$ 
	\end{proof}
\end{rem}

\begin{defn}	Let $R$ be a commutative ring, $x_0,\dots,x_n \in R$ and let $M$ be an $R$ module. We define the Koszul complex inductively as $K_\bullet(x_0,M):= \operatorname{cone}(M \xrightarrow{x_0}M)$ and $$K_\bullet(x_0,\dots,x_{n},M) := \operatorname{cone}\left(K_\bullet(x_0,\dots,x_{n-1},M) \xrightarrow{x_{n}}K_\bullet(x_0,\dots,x_{n-1},M)\right)$$ for $1 < n \leq d.$
	We define the cohomological Koszul complex  as $$K^\bullet(x_0,\dots,x_n,M):= \operatorname{Hom}(K_\bullet(x_0,\dots,x_{n},R),M).$$
\end{defn}

\begin{lem} \label{lem:koszulduality}
	Let $x_1,\dots,x_l \in R$ be a regular sequence. Then 
	\begin{enumerate}[(i)]
		\item $K_\bullet(x_1,\dots,x_l,R)$ is a projective resolution of $R/(x_1,\dots,x_l)$ as an $R$-module.
		\item There is a canonical isomorphism $$K^\bullet(x_1,\dots,x_l,M) \simeq K_\bullet(x_1,\dots,x_l,M)[-l]$$ for every $R$-module $M$.
		\item There is a canonical isomorphism $$\Rhom_R(R/(x_1,\dots,x_l),C) \simeq R/(x_1,\dots,x_l) \otimes^{\mathbb{L}}_R C[-l]$$ for every $C \in \mathbf{D}(R).$
	\end{enumerate}
\end{lem}
\begin{proof}
	For	\textit{(i)} see \cite[Corollary 4.5.5]{Weibel}. For \textit{(ii)} one first checks that $K_\bullet(x_1,\dots,x_l,M)\simeq K_\bullet(x_1,\dots,x_l,R) \otimes_R M.$ The self-duality of the Koszul complex  (cf. \cite[Proposition 17.15]{eisenbud1995commutative}) tells us $$K^\bullet(x_1,\dots,x_l,R) \simeq K_\bullet(x_1,\dots,x_l,R)[-l].$$ Since the complex $K_\bullet(x_1,\dots,x_l,R)$ consists of finitely generated free modules \textit{(ii)} follows from the fact that for every free $R$-module $N$ we have a canonical isomorphism $\operatorname{Hom}(N,R) \otimes_R - \cong \operatorname{Hom}(N,-).$ The point \textit{(iii)} follows formally from the above: By \textit{(i)} we can take $K^\bullet = K_\bullet(x_1,\dots,x_l,R)[-l]$ as a bounded projective (hence $\mathcal{K}$-flat \footnote{Since we do not assume that $C$ is bounded, we use $\mathcal{K}$-flat resp. $\mathcal{K}$-injective  resolutions for computing the derived functors. We refer to \cite[\href{https://stacks.math.columbia.edu/tag/06Y7}{Section 06Y7}]{stacks-project} resp.  \cite[\href{https://stacks.math.columbia.edu/tag/070G}{Section 070G}]{stacks-project} for the definitions.}) resolution of $R/(x_1,\dots,x_l)$ and take any $\mathcal{K}$-injective resolution $C^\bullet$ of $C.$ Now the canonical map
	$$C\otimes^\mathbb{L}_{R} K^\bullet \to \Rhom_R(\Rhom(K^\bullet,R[0]),C) = \Rhom_R(K_\bullet,C) $$
	(cf. \cite[\href{https://stacks.math.columbia.edu/tag/0A60}{Tag 0A60}]{stacks-project} for $L^\bullet = R[0]$) 
	can be seen to induce an isomorphism of the terms in 
	$$\operatorname{Tot} (C^\bullet \otimes K^\bullet) \to \operatorname{Hom}^\bullet( K_\bullet,C^\bullet)$$
	by again using that $K^\bullet$ consists of finitely generated free modules in every degree like in \textit{(ii)}. 
\end{proof}

\begin{lem} \label{lem:iterativecones} Let $R$ be a commutative ring and let $M$ be a module over $R[x_0,\dots,x_d].$ 
	Fix $k \in \{1 ,\dots d\}$ and suppose that $M$ is killed by $x_j$ for every $j\geq k.$ Set $C^\bullet := K^\bullet(x_0,\dots,x_{k-1},M)$ then $$K^\bullet(x_0,\dots,x_d,M) \simeq \bigoplus_{n=0}^{d-k+1} C^\bullet[-n]^{\binom{d-k+1}{n}}.$$
	
\end{lem}
\begin{proof}
	Since $x_k$ acts as zero on $M$ we have $$K^\bullet(x_0,\dots,x_k,M) \simeq \cone(C^\bullet \xrightarrow{0} C^\bullet)[-1]\simeq C^\bullet \oplus C^\bullet[-1].$$
	Define a sequence of complexes $Y_0^\bullet:=C^\bullet$ and $Y^\bullet_{i+1} =  Y_i^\bullet[-1] \oplus Y_i^\bullet.$ We can proceed inductively to conclude $K^\bullet(x_0,\dots,x_d,M) \simeq Y^\bullet_{d-k+1}.$ It is easy to see that $Y^\bullet_{d-k+1}$ is a direct sum of shifts of $C^\bullet$ given according to the additive inverses of the terms in the sequence $y_{d-k+1}$ from Definition \ref{def:combo}.  Indeed $ Y_1^\bullet = C^\bullet[-1] \oplus C^\bullet[0], Y_2^\bullet = C^\bullet[-2] \oplus C^\bullet[-1] \oplus  C^\bullet[-1] \oplus C^\bullet[0]$ etc..
	The shift $C^\bullet[-n]$ appears precisely $N(n,d-k+1)$ times in $Y_{d-k+1}^\bullet.$
	By Remark \ref{rem:combinatorics} we have $N(n,d-k+1)=\binom{d-k+1}{n}.$
\end{proof}
\section{Abstract Herr complexes for augmented algebras} In this section we introduce an abstract formalism to study Herr complexes, which appear in the context of $(\varphi_L,\Gamma_L)$-modules. As a consequence of this formalism one obtains spectral sequences, which allow the comparison of different cohomology theories. The applicability of these spectral sequences relies on the cohomology theories being ``representable'' in the sense that they are isomorphic to $\Rhom_{A}(*,-)$ where $A$ is some abelian category with enough projectives containing the category of $(\varphi_L,\Gamma_L)$-modules and $*$ is a suitable ``trivial object'', which is the case for all cohomology theories considered in this article. Note that these results are not limited to univariate $(\varphi_L,\Gamma_L)$-modules but are applicable in a wider generality (e.g. one could consider multivariate $(\varphi,\Gamma)$-modules or more generally modules over a commutative augmented $R$-algebra together with a choice of $n$ commuting $D$-linear operators). 
Recall that an augmented $R$-algebra is an $R$-algebra $D$ together with an $R$-algebra homomorphism $\varepsilon \colon D \to R.$ 
\begin{defn}
	Let $D$ be a commutative augmented $R$-algebra and let $M$ be a $D[\underline{T}]=D[T_1,\dots,T_n]$-module. We view $D[\underline{T}]$ as an augmented $R$-algebra by mapping each $T_i$ to $1$ and define the \textbf{abstract Herr complex} $$C_{\underline{T},D}(M):=\Rhom_{D[\underline{T}]}(R,M).$$
\end{defn}
\begin{lem} \label{lem:nonsense} Let $A$ be a commutative ring. Let $I,J \subset A$ be two ideals. Let $M$ be a an $A$-module. We have canonical isomorphisms 
	\begin{align*}\Rhom_{A}(A/(I+J),M) &\simeq \Rhom_{A/I}(A/(I+J),\Rhom_{A}(A/I,M))\\
		&\simeq  \Rhom_{A/J}(A/(I+J),\Rhom_{A}(A/J,M)).\end{align*}
\end{lem}
\begin{proof}
	The argument is symmetric in $I$ and $J,$ we thus prove that $$\Rhom_{A}(A/(I+J),M) \simeq \Rhom_{A/J}(A/(I+J),\Rhom_{A}(A/J,M)).$$
	Consider the natural map $A \to A/J.$ The functor $\operatorname{Hom}_A(A/I,-)$ takes injective $A$-modules to injective $A/I$-modules. (Indeed, suppose $L$ is an injective $A$-module then $\operatorname{Hom}_{A/I}(-, \operatorname{Hom}_{A}(A/I,L)) = \operatorname{Hom}_{A}(- \otimes_{A}A/I,L) = \operatorname{Hom}_A(-,L)$ is exact on $A/I$-modules.)
	Furthermore by identifying $$\operatorname{Hom}_A(A/(I+J),M) \cong \{m \in M \mid am=0 \text{ for all } a\in I+J\}$$
	we see that  $$\operatorname{Hom}_{A/J}(A/(I+J),\operatorname{Hom}_A(A/J,M)) = \operatorname{Hom}_A(A/(I+J),M)$$
	and the claim follows from \cite[\href{https://stacks.math.columbia.edu/tag/015M}{Lemma 015M}]{stacks-project} .
\end{proof}
\begin{rem}\label{abstract:herrcomplex} Let $D$ be a commutative augmented $R$-algebra. Let $M$ be a $D[\underline{T}]$-module then: 
	\begin{enumerate}[(i)] 
		\item	We have canonical isomorphisms
		
		\begin{align*}
			\Rhom_{D[\underline{T}]}(R,M) &\simeq \Rhom_{R[\underline{T}]}(R,\Rhom_{D[\underline{T}]}(R[\underline{T}],M)) \\
			&\simeq\Rhom_{D}(R, \Rhom_{D[\underline{T}]}(D,M)).	
		\end{align*}
		\item There is a canonical isomorphism
		$$\Rhom_{D[\underline{T}]}(R[\underline{T}],M) \simeq \Rhom_{D}(R,M).$$ 
		\item In particular (for $n=1$)
		\begin{align*}
			C_{T,D}(M) &\simeq \operatorname{cone}(\Rhom_{D}(R,M) \xrightarrow{T-1} \Rhom_{D}(R,M))  \\
			&\simeq\Rhom_{D}(R, \operatorname{cone}(M\xrightarrow{T-1}M))		
		\end{align*}
	\end{enumerate}
\end{rem}
\begin{proof} Apply Lemma \ref{lem:nonsense} to the ring $A=D[\underline{T}]$ and the ideals \\
	$I = \operatorname{ker(\varepsilon\colon D[\underline{T}] \to R[\underline{T}])}$ and $J = (T_1,\dots,T_n).$
	For \textit{(ii)} apply \cite[\href{https://stacks.math.columbia.edu/tag/0E1W}{Lemma 0E1W}]{stacks-project} to the ring map $D \to D[\underline{T}].$ Since $D[\underline{T}]$ is flat over $D$ one has $R \otimes_{D}^\mathbb{L} D[\underline{T}] \simeq R[\underline{T}].$ To see (\textit{iii}) one plugs the projective resolution \\
	$R[T] \xrightarrow{T-1} R[T]$ (resp. $D[T] \xrightarrow{T-1}D[T]$) of $R$ (resp. $D$) into the first (resp. second) quasi isomorphism in (\textit{i}) and uses (\textit{ii}).
\end{proof}

\begin{rem} \label{rem:abstractspectralsequences}
	Let $D \to D'$ be a morphism of commutative augmented $R$-algebras and let $M$ be a $D'[T]$-module. Then we have
	\begin{enumerate}[(i)]
		\item \begin{align}C_{T,D}(M) &\simeq \Rhom_{D'[T]}(R \otimes_{D[T]}^\mathbb{L}D'[T],M) \label{eq:as1}\\
			&\simeq \Rhom_{D'[T]}(R \otimes_{D}^\mathbb{L}D',M)\label{eq:as2}\\
			&\simeq \Rhom_{D[T]}(D'[T],C_{T,D'}(M))\label{eq:as3}\\
			&\simeq \Rhom_{D}(D',C_{T,D'}(M))\label{eq:as4}.\end{align}
		\item If one assumes that $D'$ is the quotient of $D$ by a regular sequence of length $l$ then $C_{T,D}(M)$ is quasi-isomorphic to $$D' \otimes_D^\mathbb{L}C_{T,D'}(M)[-l]$$
		\item There is a base change map $$C_{T,D}(M)\otimes^\mathbb{L}_{D[T]}D'[T] \to \Rhom_{D'[T]}( R \otimes^\mathbb{L}_{D[T]}D'[T],M \otimes_{D[T]}^\mathbb{L}D'[T])$$ which is an isomorphism if $R$ is perfect as a $D[T]$-module or $D'[T]$ is perfect as a $D[T]$-module. 
		
	\end{enumerate}
\end{rem}
\begin{proof}
	See \cite[\href{https://stacks.math.columbia.edu/tag/0E1V}{Tag 0E1V}]{stacks-project} for \textit{(i)} \eqref{eq:as1} and \textit{(iii)}. Equation \eqref{eq:as4} follows from \cite[\href{https://stacks.math.columbia.edu/tag/0E1W}{Lemma 0E1W}]{stacks-project}.  Equations \eqref{eq:as3} and \eqref{eq:as2} now follow by derived adjunctions from \eqref{eq:as1} and \eqref{eq:as4} respectively.
	The statement \textit{(ii)} follows from \ref{lem:koszulduality} (\textit{iii}).
\end{proof}
From the usual hyper-cohomology spectral sequence we obtain the following. 
\begin{cor} \label{cor:spectralsequencesabstract} Let $M$ be a $D'[T]$-module and $D \to D'$ a morphism of commutative augmented $R$-algebras. Then:
	\begin{enumerate}
		\item	There is a spectral sequence $$E_{2}^{p,q}=\operatorname{Ext}_D^p(D',H^q_{T,D'}(M)) \Rightarrow H^{p+q}_{T,D}(M).$$
		\item If $D'$ is the quotient of $D$ by a regular sequence of length $l$ then there is a spectral sequence
		$$E_{2}^{p,q} = \operatorname{Tor}^{D}_{l-p}(D',H^q_{T,D'}(M)) \Rightarrow H^{p+q}_{T,D}(M)$$
	\end{enumerate}
\end{cor}
\begin{rem}
	We can use the second spectral sequence from Corollary \ref{cor:spectralsequencesabstract} to provide an alternative proof of \ref{lem:iterativecones} (for $R=K$ a field).
\end{rem}
\begin{proof}
	We consider the algebras $$D = K[x_1,\dots,x_{d}] \to D'=D/(x_k,\dots,x_d) \cong K[x_1,\dots,x_{k-1}]$$ and let $T$ act on $M$ as multiplication by $x_0.$
	Let $P^\bullet$ be the Koszul complex resolving $D'$ as a $D$-module and let $C^\bullet:= C_{T,D'}(M).$ Then the spectral sequence from \ref{cor:spectralsequencesabstract} is obtained from the double complex $P^\bullet \otimes_D C^\bullet$ by first taking cohomology along the $C^\bullet$ axis leaving us with $P^\bullet \otimes_DH^i(C).$ Since each $H^i(C)$ is a $D/(x_0,\dots,x_d)$-module the differentials along the $P^\bullet$ axis are all zero (as they are linear combinations of multiplication by $x_i$-maps) and hence the spectral sequence collapses at the $E_1$ page. As $K$-vector spaces we have $H^{n}_{T,D}(M) \cong \bigoplus_{i+j=n} P^{i}\otimes_{D}H^j(C).$ Since each $P^q$ is free of rank $\binom{l}{q}$ the result follows by computing the length $l$ of the regular sequence cutting out $D'$ which in our case is $d+1-k.$
\end{proof}
\section{Applications}
\label{sec:applications}
Let $K$ be a complete field extension of a finite extension $L$ of $\QQ_p$ of degree $d = [L:\QQ_p].$ Let $\Sigma = \{\sigma_1,\dots,\sigma_d\}$ be the set of embeddings $L \to \CC_p.$ We identify $L$ with a subfield of $\CC_p$ via $\sigma_1.$ 
Let $E \subset \CC_p$ be a Galois closure of $L.$  Denote by $\operatorname{Res}(\Gamma_L)$ the restriction of scalars of $\Gamma_L$ to $\QQ_p.$ 
In \cite[after Lemma 2.6]{Berger2016} Berger defines an E-basis $(\nabla_\sigma)_{\sigma \in \Sigma}$ of $E \otimes_{\QQ_p} \operatorname{Lie}(\operatorname{Res}\Gamma_L)$ with the property that a $(\varphi_L,\Gamma_L)$-module over $\cR_E$ is $L$-analytic if and only if $\nabla_\sigma$ acts as zero on $M$ for every $\sigma \neq \id.$ We henceforth abbreviate $\nabla_i:=\nabla_{\sigma_i}.$
Let $\mathfrak{g} := \operatorname{Lie}(\Gamma_L) \cong L\nabla_{1}$ and $\mathfrak{g}_0:= E \otimes_{\QQ_p}\operatorname{Lie}(\operatorname{Res}(\Gamma_L)) \cong \bigoplus_{i=1}^d E\nabla_i.$
\subsection{Cohomology}
\begin{defn} Let $F \subset \CC_p$ be a complete extension of $L$ which contains a Galois closure $E$ of $L.$ and let $M$ be a $(\varphi_L,\Gamma_L)$-module over $\cR_F$  Let $f \in \operatorname{End}_{F}(M)$ be an operator that commutes with the action of $\mathfrak{g}_0$ on $M.$ We view  $M$ as a module over $K[x_0,\dots,x_d]$ by letting $x_0$ act as $f-1$ and $x_i$ act as $\nabla_i$ for $i=1,\dots,d.$  
	We define $$C^\bullet_{f,\mathfrak{g}}(M):=K^\bullet(x_0, x_1,M)$$ and $$C^{\bullet}_{f,\mathfrak{g}_0}(M):=K^\bullet(x_0, \dots,x_d,M).$$
\end{defn}
The above complexes are quasi-isomorphic to a shift of the cone of $f-1$ on the respective Lie algebra cohomology with values in  $M.$ As a consequence of \ref{thm:LieundAN} and the preceding discussions the groups $H^i_{cts}(M):=H^0(\Gamma_L,H^i_{\varphi,\mathfrak{g_0}}(M)$ (resp. $H^i_{an}(M):=H^0(\Gamma_L,H^i_{\varphi,\mathfrak{g}}(M)$)) are isomorphic to the cohomology of $$\operatorname{cone}(\Rhom_{D^{\QQ_p\text{-an}}(\Gamma_L,K)}(K,M) \xrightarrow{\varphi_L-1}\Rhom_{D^{\QQ_p\text{-an}}(\Gamma_L,K)}(K,M))[-1]$$ (resp. for the $L$-analytic variant).
As a corollary of Lemma \ref{lem:iterativecones} we get the following comparison between cohomologies of $M.$
\begin{thm} \label{thm:ankoszul}
	Let $M$ be an $L$-analytic $(\varphi_L,\Gamma_L)$-module over $\cR_F$ and let $f \in \operatorname{End}_L(M)$ be an operator that commutes with the action of $\mathfrak{g}_0$ on $M.$
	\begin{enumerate}[(i)]
		\item  $C^{\bullet}_{f,\mathfrak{g}_0}(M)  \simeq \bigoplus_{n=0}^{d-1} C^{\bullet}_{f,\mathfrak{g}}(M)[-n]^{\binom{d-1}{n}}.$
	\end{enumerate}
\end{thm}
\begin{proof}
	Apply Lemma \ref{lem:iterativecones} with $k = 2.$ 
\end{proof}
\begin{rem} By applying  Lemma \ref{lem:iterativecones} for larger $k$ one can obtain a similar result for (not nescessarily analytic) $(\varphi_L,\Gamma_L)$-modules on which a finite subset of $\{\nabla_{\sigma},\sigma \in \Sigma \}$ acts trivially.
\end{rem}
We get a generalisation of \cite[Theorem 0.2]{FX12} as a corollary.
\begin{cor} 	\label{cor:FX} Let $M$ be an $L$-analytic $(\varphi_L,\Gamma_L)$-module over $\cR_K$ where $L\subset K$ is a complete field extension (not necessarily containing a normal closure of $L$). Then
	we have an isomorphism of $F$-vector spaces
	\begin{equation}\label{eq:iso} H^i_{cts}(M) \cong H^0_{an}(M)^{\binom{d-1}{i}} \oplus H^1_{an}(M)^{\binom{d-1}{i-1}}  \oplus H^2_{an}(M)^{\binom{d-1}{i-2}},\end{equation}
	for every $i \geq 0.$ In particular:\\
	$H^1_{cts}(M) \cong H^1_{an}(M) \oplus H^0_{an}(M)^{[L:\mathbb{Q}_p]-1}$ and $H^{d+1}_{cts}(M) \cong H^2_{an}(M).$
\end{cor}
\begin{proof} First assume that $K$ contains a normal closure of $L.$
	The first part follows from Theorem \ref{thm:ankoszul} for $f= \varphi_L$ by taking cohomology and $\Gamma_L$-invariants. For the second part one uses $N(1,d-1) = d-1, N(d+1,d-1)=N(d,d-1)=0$ and $N(0,d-1)=N(d-1,d-1)=1.$
	For the general case
	let $E$ be a normal closure of $L$ inside $\CC_p$ and let $F=KE.$  Let us consider the base change map from Remark \ref{rem:abstractspectralsequences} \textit{(iii)} where $D' = F \otimes_K D$ and $D$ is an augmented $K$-algebra (to be chosen below). The base change map is an isomorphism since $D'$ is finite free over $D.$ Since $D'$ is furthermore flat we have that all derived tensor products are represented by the usual tensor product. 
	In particular $F \cong K \otimes^{\mathbb{L}}_D D'.$ Let $M$ be a $D[T]$-module and let $N = D'\otimes_DM.$ 
	Then $$\Rhom_{D[T]}(K,N) \simeq \Rhom_{D'[T]}(F,N) \simeq D' \otimes_D   \Rhom_{D[T]}(K,M).$$
	The cohomology groups on the left-hand side carry a natural action of $G(F/K)$ via $N$ and $G(F/K)$ acts on the right-hand side by acting on $D'.$ The induced map 
	$$H^i_{T,D'}(N) \xrightarrow{\cong} D' \otimes_D H^i_{T,D}(M) \cong F\otimes_K H^i_{T,D}(M) $$ is equivariant with respect to these actions and taking $G$-invariants shows $H^i_{T,D}(M) \cong (H^i_{T,D'}(N))^{G}.$
	Now let $M$ be an $L$-analytic $(\varphi_L,\Gamma_L)$-module over $\cR_K$ and let $N:= F \otimes_K M.$ We can apply \eqref{eq:iso} to $N,$ reinterpret the continuous (resp. analytic) cohomology in terms of abstract Herr complexes for $D = D^{\QQ_p{\text{-an}}}(\Gamma_L,F)$ (resp.  $D^{L{\text{-an}}}(\Gamma_L,F))$ and then use the preceding considerations to get the same result for the cohomology of $M.$

\end{proof}
In \cite[Question 5.2.12]{SchneiderVenjakobRegulator} Schneider and Venjakob ask, whether for an étale $(\varphi_L,\Gamma_L)$-module $M,$ and an open subgroup $U\subset\Gamma_L$ the continuous cohomology groups $H^i_{cts}(\varphi_L^\mathbb{N} \times U,M)$ vanish outside of degree $[0,2]$ and whether they are finite dimensional. Theorem \ref{thm:ankoszul} shows that in general we should not expect the Lie algebra cohomology to be concentrated in $[0,2].$ If $M$ is analytic, then $H^0(\Gamma_L,H^i_{\varphi,\mathfrak{g}}(M)) \neq 0$ for some $i \in [0,2]$ implies $H^0(\Gamma_L,H^{d+i-1}_{\varphi,\mathfrak{g}_0}(M)) \neq 0$ by looking at the three highest cohomology groups in Theorem \ref{thm:ankoszul}. A similar statement holds if one replaces $\varphi_L$ by another operator commuting with $\Gamma_L$ (for example the left-inverse $\Psi$-operator). More explicitly, one can for $d>2$ consider the rank one module attached to the Lubin-Tate character $\chi_{LT},$ which by \cite[Theorem 5.19]{FX12} has non-trivial $H^1_{an}$ and thus has non-trivial $H^d_{cts}.$
\subsection{Euler characteristics} Let $M$ be an $L$-analytic $(\varphi_L,\Gamma_L)$-module over $\cR_K$.
In this subsection we assume that $H^i_{an}(M)$ is finite dimensional and denote its dimension by $h^i_{an}.$\footnote{The finiteness is known for large $K$ due to \cite{steingart2022finiteness}.}
Then by Corollary \ref{cor:FX} $H^i_{cts}(M)$ is finite dimensional as well and we denote its dimension by $h^i_{cts}(M).$ For $* \in \{\text{an},\text{cts}\}$ we denote the corresponding Euler characteristic by $\chi_{*}:=\sum_{i} (-1)^ih^i_{*}.$
\begin{rem}\label{rem:Nchi} Let $N_\chi:= \sum_{k=0}^{d-1}(-1)^kN(k,d-1).$ If $d=1,$ then $N_\chi=1.$ Otherwise $N_\chi=0.$
\end{rem}
\begin{proof}
	Let $V$ be a one-dimensional $K$-vector space and let $C_0:= V[0].$ We define inductively $C_{i+1}:= C_i \oplus C_i[-1].$ It is easy to see that $N_\chi$ is the Euler-characteristic  of the complex $C_{d-1}$. It is $1,$ if $d=1$ and otherwise the Euler-characteristics of $C_{i}$ and $C_i[-1]$ cancel out for all $0 \leq i \leq d-2$ and by induction $C_{d-1}$ has Euler-characteristic $0.$ 
\end{proof}
\begin{rem} Assume $\chi_{an}$ is finite. Then
	$$\chi_{cts} = N_\chi \chi_{an}.$$
\end{rem}
\begin{proof}
	Redo the proof of Remark \ref{rem:Nchi} with $C_0 = C_{f,\mathfrak{g}}(M).$
\end{proof}
\begin{rem}
	Let $M$ be an $L$-analytic $(\varphi_L,\Gamma_L)$-module over $\cR_{K},$ such that $\chi_{an}(M) \neq 0.$ Suppose  $L \neq \QQ_p$ then
	\begin{enumerate}[(i)]
		\item $C_{an}(M)$ and $C_{cts}(M)$ are not quasi-isomorphic.
		\item If $M$ is the base change of an étale $(\varphi_L,\Gamma_L)$-module over $\mathcal{R}_L$ attached to $V \in \operatorname{Rep}_L(G_L),$ then $C_{cts}(M)$ is not quasi-isomorphic to a base-change of $C^\bullet_{cts}(G_L,V)$ with respect to any $L$-algebra.
	\end{enumerate}
\end{rem}
\begin{proof}
	(i) and (ii) follow from the fact that both $C_{an}(M)$ and $C_{cts}(G_L,V)$ have non-zero Euler characteristic but $N_\chi=0$ and since any $L$-algebra is flat and the base change preserves the Euler-Poincaré-characteristic.
\end{proof}
Note that examples of modules with non-zero $\chi_{an}$ are provided by Colmez in \cite{colmez2016representations} in the rank one case provided that $K$ is large enough (see also \cite{steingart2022iwasawa}). 
The above result is surprising at first glance as one would expect $C^\bullet_{cts}(\mathbf{D}^\dagger_{rig}(V))$ to be isomorphic to Galois cohomology and hence concentrated in $[0,2]$ with Euler characteristic $-[L:\QQ_p]\operatorname{dim}_L(V).$ From the classical case we know that the value of the Euler characterstic can be expressed as $-\operatorname{rank}_{o_L\llbracket \Gamma_L\rrbracket[1/p]}(H^1_{Iw}(V)),$ which gives us a different perspective on this problem and heuristically explains, why the Euler characterstic of $C_{cts}(M)$ is zero: If $M$ is analytic, then the analogue of the first Iwasawa cohomology, i.e.\ $M^{\Psi=1},$ is a $D(\Gamma_L,L)$-module - in particular torsion when viewed as a module over the algebra of $\QQ_p$-analytic distributions and hence of rank $0$.

\subsection{Cup products}
Corollary \ref{cor:FX} makes it interesting to investigate where the non-trivial elements of $H^2_{cts}(M)$ for an analytic $(\varphi_L,\Gamma_L)$-module $M$  come from, if, say, $H^2_{an}(M)$ is trivial. We show that cup products with elements of $H^1_{cts}(\cR_L) \setminus H^1_{an}(\cR_L)$ are a source of such elements provided that $H^0(M)=0.$ (Since $H^0_{an}(M)=H^0_{cts}(M)$ we write simply $H^0(M).$) In this subsection assume $R=K \subset \CC_p$ is a closed subfield containing a normal closure of $L.$

\begin{defn}
	Let $M$ be a $D[T]$-module. We define the cup product 
	$$\cup \colon H^i_{T,D}(M) \times H^j_{T,D}(K) \to H^{i+j}_{T,D}(M)$$ via the Yoneda product in the category of $D[T]$-modules.
\end{defn}
Recall that for the continuous (resp. analytic) cohomology groups of a (analytic) $(\varphi_L,\Gamma_L)$-module we have  $$ H^1_{cts}(M)\cong H^i_{T,D}(M)$$ resp. $$H^1_{an}(M)\cong H^i_{T,D'}(M),$$ where $D = D^{\QQ_p\text{-an}}(\Gamma_L,K)$ (resp. $D' = D^{L\text{-an}}(\Gamma_L,K)$) and $T$ acts as $\varphi_L$. At the same time, the first continuous (resp.\ analytic) cohomology group parametrises extensions $0 \to M \to E \to \cR_K \to 0$ of $(\varphi_L,\Gamma_L)$-modules over $\cR_K$ (resp. analytic $(\varphi_L,\Gamma_L)$-modules over $\cR_K.$) (cf. \cite[Theorem 0.1]{FX12}).
\begin{rem}
	The natural map $K \to \cR_K$ induces isomorphisms
	$$H^1_{T,D}(K) \cong H^1_{T,D}(\cR_K)$$
	and
	$$H^1_{T,D'}(K) \cong H^1_{T,D'}(\cR_K).$$
\end{rem}
\begin{proof} Using \ref{thm:LieundAN} and \ref{cor:FX} it suffices to show the statement for the $L$-Lie-algebra variant. 
	This is done in \cite[Proposition 5.10 (b)]{FX12} (cf. also \cite[Remarque 5.17]{colmez2016representations} for a similar statement over a large field).
\end{proof}
As a consequence of the above we can now describe the cup product of a  $1$-extension of $(\varphi_L,\Gamma_L)$-modules with a $1$-extension of $\cR_K,$ by taking the corresponding elements in $\operatorname{Ext}^1_{D[T]}(K,M),\operatorname{Ext}^1_{D[T]}(K,K)$ (resp. for $D'$) and composing them to get a $2$-extension $\operatorname{Ext}^2_{D[T]}(K,M) \cong H^2_{cts}(M).$ We are not aware of an evident description of the elements of $H^2_{cts}(M)$ in terms of $2$-extensions of $(\varphi_L,\Gamma_L)$-modules over $\cR_K$.  
\begin{rem} \label{prop:descriptionofcup} Let $?$ be either $\text{an}$ or $\text{cts}.$ For a (analytic if $? = \text{an}$) $(\varphi_L,\Gamma_L)$-module $M$ over $\cR_K$ 
	consider the pairing 
	$$H^1_{?}(M) \times H^n_{?}(K) \to H^{n+1}_{?}(M).$$
	Let $\xi \in H^1_{?}(M)$ and let $v \in H^n_{?}(K)$ then $\xi \cup v = \delta(v)$ under the map $H^n_{?}(K) \xrightarrow{\delta} H^{n+1}_{?}(M)$ induced by $\xi.$
	
\end{rem}
\begin{proof}
	See \cite[Chapter VII, Section 5]{mitchell1965theory}.
\end{proof}

\begin{lem}  \label{lem:cupkernel} Let $M$ be an analytic $(\varphi_L,\Gamma_L)$-module satisfying $H^0(M)=H^2_{\text{an}}(M)=0.$  Let $E \in \operatorname H^1_{an}(M)$ be a non-split extension. Then $$\operatorname{Ker}(H^1_{cts}(\cR_K) \xrightarrow{\delta} H^2_{cts}(M)) = H^1_{an}(\cR_K).$$ 
\end{lem}
\begin{proof}
	Suppose $v \in H^1_{cts}(\cR_K)$ with $\delta(v) =0.$ This means $v$ is in the image of $H^1_{cts}(E) \to H^1_{cts}(\cR_K)$ but from $E$ being non-split we can deduce $H^0(E)=0$ (as any non-zero element $e \in H^0(E)$ would define a split of the sequence by mapping $1 \in K$ to $e$). From \ref{cor:FX} it follows that $H^1_{cts}(E)= H^1_{an}(E)$ and thus $v \in H^1_{an}(\cR_K).$ If $v \in H^1_{an}(\cR_K)$ then  $\delta(v) = 0$ since $\delta$ factors over the natural map $0=H^2_{an}(M) \to H^2_{cts}(M).$
\end{proof}

\begin{thm} \label{thm:surjectivecup} Suppose $M$ is an analytic $(\varphi_L,\Gamma_L)$-module satisfying $$H^0(M)=H^2_{\text{an}}(M)=0$$ and assume $\operatorname{dim}_K(H^1_{an}(M))=1.$\footnote{Under the assumption of the analytic Euler-Poincaré formula (i.e. $\sum (-1)^i\operatorname{dim}H^i_{an}(M)= -\operatorname{rank}(M)$) this is equivalent to $M$ being of rank $1.$ This formula is known in the rank $1$ case for big $K$ (cf. \cite[Remark 6.3]{steingart2022iwasawa}).}
	Then the map 
	$$\cup \colon H^1_{\text{an}}(M) \otimes H^1_{\text{cts}}(\cR_K) \to H^{2}_{\text{cts}}(M)$$ is surjective and induces an isomorphism $$ H^1_{\text{an}}(M) \otimes H^1_{\text{cts}}(\cR_K)/H^1_{\text{an}}(\cR_K) \cong  H^{2}_{\text{cts}}(M)$$.
\end{thm}
\begin{proof} 
	From \ref{lem:cupkernel} we know that for every $0 \neq \xi \in H^1_{an}(M)$ the induced map $H^1_{cts}(\cR_K)/H^1_{an}(\cR_K) \to H^2_{cts}(M), v \mapsto v \cup \xi$ is injective. In particular the image is a $[L:\QQ_p]-1$-dimensional subspace of $H^2_{cts}(M).$ The statement follows for dimension reasons since we know from \ref{cor:FX} that $H^2_{cts}(M)$ is of dimension $([L:\QQ_p]-1)\operatorname{dim}(H^1_{an}(M)) =[L:\QQ_p]-1 .$ Using that $H^1_{an}(\cR_K)$ is $2$-dimensional (cf. \cite{FX12}) we deduce by similar arguments that $H^1_{cts}(\cR_K)$ is of dimension $([L:\QQ_p]-1)\operatorname{dim}(H^0_{an}(K)) + \operatorname{dim}(H^1_{an}(K)) = [L:\QQ_p]-1+2 = [L:\QQ_p]+1.$
\end{proof}
\begin{thm} \label{thm:nondegcup} Let $M$ be analytic and with the property $$H^0_{an}(M) = H^2_{an}(M)=0.$$
	The pairing $$ H^1_{\text{an}}(M) \otimes H^1_{\text{cts}}(\cR_K)/H^1_{\text{an}}(\cR_K) \to H^2_{cts}(M)$$ is non-degenerate. 
\end{thm}
\begin{proof} First let $E \in H^1_{an}(M) \neq 0.$
	Then the proof of \ref{lem:cupkernel} shows that the induced map $H^1_{cts}(\cR_K)/H^1_{an}(\cR_K) \to H^2_{cts}(M)$ is even injective (in particular non-zero).	
	
	Let $E' \in H^1_{cts}(\cR_K)$ and let $E \in H^1_{an}(M)$ such that $E \cup E' = 0.$ Write $D:= D^{\QQ_p\text{-an}}(\Gamma_L,K)$ and $D':= D^{L\text{-an}}(\Gamma_L,K).$
	By \cite[Lemma 4.1]{mitchell1965theory} $E'$ fits into an exact sequence of $D[T]$-modules $0 \to M \to F \to E' \to 0$ with some $F \in \operatorname{Ext}_{D[T]}(K,E).$ If $E$ is non-split we have $\operatorname{Ext}_{D[T]}(K,E) = \operatorname{Ext}_{D'[T]}(K,E)$ by \ref{cor:FX}. This expresses $E'$ as a quotient of a $D'$-module. In particular $E' \in H^1_{\text{an}}(K)$ as desired.
\end{proof}
The quotient $H^1_{\text{cts}}(\cR_L)/H^1_{\text{an}}(\cR_L)$ can be interpreted in terms of Galois cohomology. 
We denote by $E_{un} \in H^1_{cts}(G_{\QQ_p},L) \cong \operatorname{Hom}_{cts}(G_{\QQ_p}^{ab},L)$ the unramified extension of $L$ corresponding to the valuation $\QQ_p^\times \to L$ by local class field theory and denote by $E_{\QQ_p}$ the extension $0 \to L \to E' \to L \to 0$ on which $g \in G_{\QQ_p}$ acts via  $$\begin{pmatrix} 1 & \log (\chi_{cyc}(g)) \\
	0 & 1
\end{pmatrix}.$$
Similary we can define an extension $E_{LT}  \in H^1_{cts}(G_L,L)$ $$\begin{pmatrix} 1 & \log (\chi_{LT}(g)) \\
	0 & 1
\end{pmatrix}.$$ These extensions are always non-split and not contained in the span of $E_{un}$ (resp. $\operatorname{res}(E_{un})$), as we will see shortly.
By interpreting $H^1_{cts}(\cR_L)$ in terms of extensions of $(\varphi_L,\Gamma_L)$-modules and observing that étale objects are stable under extensions we can view $H^1_{cts}(\cR_L)$ as a subspace of the Galois cohomology group $H^1_{cts}(G_L,L).$ By similar reasoning $H^1_{an}(\cR_L)$ can be viewed as a subspace of $H^1_{cts}(\cR_L).$
It follows from the proof of \ref{thm:surjectivecup} that  $H^1_{cts}(\cR_L)$ is a $[L:\QQ_p]+1$-dimensional subspace of $H^1_{cts}(G_L,L).$ A small computation using Tate-Duality and the Euler-Poincaré formula for Galois cohomology shows that the latter is of the same dimension and hence $H^1_{cts}(\cR_L) \cong H^1_{cts}(G_L,L).$
\begin{thm} 
	Let $L\neq \QQ_p$ and assume $L/\QQ_p$ is Galois. Under the above isomorphism $H^1_{cts}(\cR_L) \cong H^1_{cts}(G_L,L)$ the subspace $H^1_{an}(\cR_L)$ is spanned by $E_{LT}$ and the image of $E_{un}$ under the restriction map. $E_{\QQ_p}$ is not $L$-analytic.
\end{thm}
\begin{proof}
	Since the $(\varphi_L,\Gamma_L)$-action (resp. the Galois action) on $L$ is trivial $H^1_{cts}(L) \cong H^1_{cts}(\cR_L)$ can be identified with $\operatorname{Hom}_{cts}(L^\times,L) \cong \operatorname{Hom}_{cts}(G_L^{ab},L)$ by local class field theory. The injectivity of $$H^1_{cts}(\QQ_p,L) \to H^1_{cts}(G_L,L)$$ follows from the inflation-restriction exact sequence since $L$ is a $\QQ$-vector space.
	It is easy to see, that the valuation $ v_p\colon \Q_p^\times \to L$ defines an unramified non-split extension of $L$ by itself, which is in particular Hodge-Tate at every embedding and, since the weights add to $0$ at every embedding the representation $\operatorname{res}(v_p)$ is analytic. One computes that $E_{\QQ_p}$ has Sen operator\footnote{Cf. \cite[3.2.3]{Fontaine} for the definition of the operator. } $$\begin{pmatrix}0 & 1 \\
		0 & 0
	\end{pmatrix} $$ and concludes that $E_{\QQ_p}$ is not Hodge-Tate using \cite[Proposition 6.5]{Fontaine} because the above operator is not semi-simple\footnote{I found this example in unpublished notes \cite{Porat} of Gal Porat, which used to be available on his webpage.} This remains true for the restriction of $E_{\QQ_p}$ at every embedding of $L,$ since the cyclotomic character has Galois stable image.
	By a similar argument the extension $E_{LT}$ has Sen operator $$\begin{pmatrix}0 & 1 \\
		0 & 0
	\end{pmatrix}$$ at the identity embedding, but, since $\chi_{LT}$ has Hodge-Tate weight $0$ at the non-identity embeddings, we can conclude that the Sen operator at the other embeddings vanishes. In particular $E_{LT}$ is $L$-analytic (but not Hodge-Tate). 
	The subset of $H^1_{cts}(L)$ consisting of Hodge-Tate extensions can be identified with the kernel of the natural map $H^1(G_L,L) \to H^1(G_L,B_{HT})$ and is hence a sub $L$-vector space. We conclude that $E_u$ and $E_{LT}$ are linearly independent. By \cite[Theorem 5.10]{FX12} $\dim_L(H^1_{an}(L))=2$ and hence $E_u$ and $E_{LT}$ form a basis of $H^1_{an}(L).$ Lastly $\operatorname{res}(E_{\QQ_p})$ can not be contained in  $H^1_{an}(L),$ as it would have to be Hodge-Tate at the non-identity embeddings. 
\end{proof}
\section{A Dolbeaut complex for $p$-adic Galois representations}

\subsection{The complex analogy}
Consider the following example from complex geometry. 
Let $X$ be the complex plane and $N = C^\infty(X)$ the space of smooth $\CC$-valued functions. Then inside of $N$ we find $M:=\mathcal{O}(X)$ cut out by the Cauchy-Riemann differential equation $\overline{\partial}.$ 
Let us write $\operatorname{Sol}(N) := N^{\overline{\partial}=0} (= M).$ In this for $(i,j) \in \{(1,0),(0,1),(1,1)\}$ the spaces $\Omega^{(i,j)}$ of $(i,j)$-forms are isomorphic to $N$ and the Dolbeaut double complex 
$$\begin{tikzcd}
	\Omega^{(0,0)}(X)  \arrow[r, "\overline{\partial}"] \arrow[d, "\partial"] & {\Omega^{(0,1)}(X) } \arrow[d, "\partial"'] \\
	{\Omega^{(1,0)}(X) } \arrow[r, "\overline{\partial}"]               & {\Omega^{(1,1)}(X) }                       
\end{tikzcd}$$
is just:
$$\begin{tikzcd}
	N \arrow[r, "\overline{\partial}"] \arrow[d, "\partial"] & N \arrow[d, "\partial"'] \\
	N \arrow[r, "\overline{\partial}"]                       & N                       
\end{tikzcd}.$$
The total complex  
$$0 \to E^0(X) \xrightarrow{d}E^1(X) \to \dots $$ of this double complex computes de Rham cohomology $H^{n}_{dR}(X)$. 
The double complex spectral sequence produces the Dolbeaut to de Rham spectral sequence 
$$E_1^{p,q}(X) = \frac{\operatorname{ker}(\overline{\partial} \colon \Omega^{p,q} \to \Omega^{p,q+1})}{\operatorname{im}(\overline{\partial}\colon  \Omega^{p,q-1}  \to \Omega^{p,q})} \implies H^{p+q}_{dR}(X).$$ In our case $\overline{\partial}$ is surjective and hence
$$ E_1^{p,q}(X)  =    \begin{cases}
	\operatorname{Sol}(N), & \text{if}\ q=0 \text{ and } p=0,1 \\
	0, & \text{otherwise}
\end{cases}.$$ 
Moreover, the $E_1$ page is just $	\operatorname{Sol}(N) \xrightarrow{\partial}\operatorname{Sol}(N).$
Philosophically, the comparison results of the preceding section were not the right perspective, since we were trying to compare (in our analogy) the cohomology of $$M\xrightarrow{\partial}M$$ with the cohomology of the total complex $$\begin{tikzcd}
	M\arrow[r, "\overline{\partial}"] \arrow[d, "\partial"] & M \arrow[d, "\partial"'] \\
	M \arrow[r, "\overline{\partial}"]                       & M.                      
\end{tikzcd}$$ Instead we should be looking for some bigger object $N$ such that $M=\operatorname{Sol}(N)$ is the subset of $L$-analytic vectors of $N$ and compare a cohomology attached to $N$ with the analytic cohomology of $\operatorname{Sol}(N)$ in a suitable sense.
\subsection{Dolbeaut resolutions of analytic vectors} In this section let us assume that $L/\QQ_p$ is Galois and we denote by $\Sigma = \{\id= \sigma_1,\dots,\sigma_d\}$ the embeddings $L \to \CC_p.$ Set $\Sigma_0:= \Sigma \setminus\{\sigma_1\}.$
To make the analogy described above precise we will work over the ring $B:=(\tilde{\mathbf{B}}^\dagger_{\operatorname{rig},L})^{pa}$ from \cite{Berger2016}.
In order to keep the notation in this section consistent with the rest of the article we adjust all intervals to be given terms of absolute values rather than valuations and write everything in a way which does not require fixing a normalisation of the norm on $\CC_p.$ Implicitly we endow $\CC_p^\flat$ with the norm satisfying $\abs{x^\#} = \abs{x}_{\flat}$ for all $x \in \CC_p^\flat$ (cf. \cite[Section 1.4]{Schneider2017}). Let $w$ be the Teichmüller lift of a pseudouniformiser of $o_{\CC_p^\flat}.$
Any element of $f \in W(o_{\CC_p^\flat})_L[1/\pi_L,1/w]$ can be written uniquely as a convergent series $$f= \sum_{k\gg -\infty}\pi_L^k[x_k]$$ with $x_k$ a bounded sequence in $\CC_p^\flat.$ For $0<r<1$ define $$\abs{f}_r := \sup_k \abs{\pi_L}^k(\abs{x_k}_\flat)^{-\log_p(r)}$$ and $\abs{f}_I := \sup_{r \in I} \abs{f}_r.$ \footnote{The precise value $\abs{f}_r$ depends on the chosen normalisation, but an expression of the form $\abs{f}_r = \abs{x}$  for some $x \in \CC_p$ is independent of the normalisation. } For a closed interval $I \subseteq (0,1)$ we denote by $\tilde{\fB}^I$ the completion of $ W(o_{\CC_p^\flat})_L[1/\pi_L,1/w]$ with respect to $\abs{-}_I.$ The subspace $\tilde{\fB}^I_L := (\tilde{\fB}^I)^{H_L}$ is then a Banach space representation of $\Gamma_L$ and we denote by $B^I:= (\tilde{\fB}^I_L)^{la}$ the subspace consisting of locally analytic vectors which is the set of elements $v \in \tilde{\fB}^I_L$ such that the orbit map $\Gamma_L \to \tilde{\fB}^I_L$ mapping $\gamma$ to $\gamma v$ is locally $\QQ_p$-analytic.\\
For the Fréchet spaces $\tilde{\fB}^{[r,s)} = \varprojlim_{r<s'<s}\tilde{\fB}^{[r,s']}$ we denote by $B^{[r,s)}$ the subspace $(\tilde{\fB}_L^{[r,s)})^{pa}$ of $\tilde{\fB}_L^{[r,s)} =(\tilde{\fB}^{[r,s)})^{H_L}$ consisting of pro-analytic vectors, i.e., elements $v \in \tilde{\fB}_L^{[r,s)}$ whose image in $\tilde{\fB}_L^{[r,s']}$ is analytic for every $r<s'<s.$
We give a short summary of results we need from \cite{Berger2016}. Berger studies $F$-linear representations of $G_K,$ and for technical reasons introduces a base change to a normal closure $E$ of $F.$ To keep notation simpler we will assume $L=E=F=K$ and we will be studying (analytic) $L$-linear $G_L$-representations. 
Berger attaches to an $L$-linear representation $V$ of $G_L$ a $(\varphi_L,\Gamma_L)$-module over $B^{[r,1)}$ for large enough $0<r<1$ (given by the pro-analytic vectors in $\tilde{D}^{\dagger,r}_{rig,L}(V) :=(\tilde{\fB}^{[r,1)} \otimes_L V)^{H_L}$).
He constructs differential operators $\partial_{\sigma} \colon B^I \to B^I$ which should be seen as analogues of the Dolbeaut operator in complex geometry.
To this end he defines elements $y_\sigma$ for each $\sigma \in \Sigma$ such that $y_{\id}$ is the usual ``variable'' with $\varphi_q(y_{\id}) = [\pi_L](y_{\id})$ and $\gamma y_{\id}= [\chi_{LT}(\gamma)](y_{\id})$ and the action on the $y_\sigma$ is given by twisting the coefficients with the embedding $\sigma.$ For each $\sigma \in \Sigma$ we have a corresponding element $t_\sigma := \operatorname{log}_{LT}^{\sigma}(y_\sigma),$ where the superscript $\sigma$ denotes applying the embedding $\sigma$ to the coefficients. We define $\partial_\sigma:= \frac{ g_\sigma }{t_\sigma}\nabla_\sigma,$ where $g_\sigma \in B^I$ is the formal derivative of $t_\sigma$  with respect to $y_\sigma.$ One can show that $\nabla_\sigma \colon B^I \subset t_\sigma B^I.$ The $L$-analytic vectors $(B^{I})^{L\text{-an}}$ are given by the subspace $(B^I)^{\partial_2,\dots,\partial_n=0},$ where $\partial_i:= \partial_{\sigma_i}.$

We summarise the main technical results.
\begin{thm}\label{thm:monodromy}
	The following hold:
	\begin{enumerate} 
		\item 	Let $M$ be a free $B^I$-module with a compatible $\Gamma_L$-action such that $\nabla_{\tau}(M) \subset t_{\tau}M$ for every $\tau \in \Sigma_0.$ Then $$\operatorname{Sol}(M):= \{x \in M \mid \partial_{\tau}(x)=0 \text{ for all } \tau \in \Sigma_0\}$$ is a free $(B^{I})^{L\text{-an}}$-module and $$M = B^{I}\otimes_ {(B^{I})^{L\text{-an}}}\operatorname{Sol}(M).$$
		\item Let $V \in \operatorname{Rep}_L(G_L),$ then $(\tilde{D}^{\dagger,r}_{rig,L}(V))^{pa}$ is a free $(\tilde{\fB}^{[r,1)})^{pa}$-module of rank $\dim(V)$ stable under $\Gamma_L$ for $0<r<1$ large enough.
		\item If $V \in \operatorname{Rep}_L(G_L)$  is $\CC_p$-admissible at $\tau \in \Sigma,$ then $$\nabla_{\tau}(\tilde{D}^{\dagger,r}_{rig,L}(V))^{pa}\subset t_\tau(\tilde{D}^{\dagger,r}_{rig,L}(V))^{pa}.$$
	\end{enumerate}
\end{thm}
\begin{proof}
	The first point follows from the proof of \cite[Theorem 6.1]{Berger2016}. The second part is essentially Theorem $A$ in (loc.cit.). The third part is Lemma 10.2 in (loc.cit.).
\end{proof}

We will prove an analogue of the Dolbeaut-Lemma for the operators $\partial_{i}, i\neq 1.$ A technical inconvenience will be the fact that we do not assume that $M$ is stable under $\partial_{1}.$ However, $M$ is always stable under $\nabla_1.$ Another inconvenience is the fact that $\partial_i$ does not commute with $\varphi_L$ so we need to adjust our Dolbeaut complexes to make them $\varphi_L$-equivariant. Before doing so, we first study the resolution of $B = (\tilde{\fB}_L^I)^{la}$ itself.

\begin{lem}[$\overline{\partial}$-Lemma] \label{lem:poincare}
	Let $L\neq \QQ_p$ then:
	\begin{enumerate}[(i)]
		\item The operators $\partial_{\tau} \colon B \to B$ are surjective for every $\tau \in \Sigma.$
		\item If $d=[L:\QQ_p]>2$ and $x \in B$ satisfies $\partial_{\tau}x=0$ for all $\tau \in \Xi \subset \Sigma \setminus \{\id,\sigma\},$ where $\sigma \neq  \id,$ then we can find $z$ satisfying $\partial_\tau z=0$ for all $\tau \in \Xi $ and $\partial_\sigma z=x.$ 
	\end{enumerate}
	\begin{proof}
		Part \textit{(i)} is \cite[Corollary 5.5]{Berger2016}.
		For part \textit{(ii)} we recall that Berger explicitly defines a series (using the notation from (loc.cit.)) $$z = \sum_{\mathbf{i} \in \NN_0^{\Sigma_0}} \frac{x_{\mathbf{i}}}{\mathbf{i}_\sigma+1}(\mathbf{y}-\mathbf{y}_n)^{\mathbf{i}+1_\sigma}$$ satisfying $\partial_\sigma z=x,$ where $x_\mathbf{i}$ is given by
		$$x_\mathbf{i} := \frac{1}{\mathbf{i}!}\sum_{\mathbf{k} \in \NN_0^{\Sigma_0}}(-1)^\mathbf{k}\frac{(\mathbf{y}-\mathbf{y}_n)^{\mathbf{k}}}{\mathbf{k}!}\partial^{\mathbf{k}+\mathbf{i}}(x).$$
		Berger shows $\partial_{\tau}x_\mathbf{i} =0$ for all $\tau \neq \id$ and further shows $\partial_\tau((\mathbf{y}-\mathbf{y}_n)^\mathbf{k})=0$ if $\mathbf{k}_\tau =0,$ where $\mathbf{k}_\tau$ denotes the $\tau$-component of $\mathbf{k}.$
		The assumptions of (\textit{ii}) guarantee $x_\mathbf{i}=0$ as soon as $\mathbf{i}_\tau \neq 0$ for the embeddings $\tau \in \Xi .$ If $\mathbf{i}_\tau= 0$ for $\tau \in \Xi $ then, because $\sigma \notin \Xi ,$ we also have $(\mathbf{i}+1_\sigma)_\tau=0$ and hence $\partial_\tau((\mathbf{y}-\mathbf{y}_n)^{\mathbf{i}+1_\sigma})=0.$
		Combining all of the above shows $\partial_{\tau}z=0$ for all $\tau \in \Xi.$
	\end{proof}
	We remark that it is important to have at least two auxiliary embeddings for Lemma \ref{lem:poincare} to work.
\end{lem}
\begin{lem}[Dolbeaut Lemma] \label{lem:dolbeaut} Let $L\neq \QQ_p$. 
	Let $B:= (\tilde{\fB}_L^I)^{la}$. Then the natural map
	$B^{L-an}[0]  \to K^\bullet(\partial_2,\dots,\partial_d,B)$ is a quasi-isomorphism. 
\end{lem}
\begin{proof}
	Because $\partial_2$ is surjective we have a short exact sequence 
	$$E_2: 0 \to B^{\partial_2=0} \to B \xrightarrow{\partial_2} B \to 0.$$ 
	For $n\geq 2$ we proceed inductively. Since $E_2$ is exact we have that
	$E_3':=\operatorname{Cone}(\partial_3 \colon E_2 \to E_2)[-1]$ is exact and of the form $$0 \to B^{\partial_2=0} \xrightarrow{\partial_3,\id} B^{\partial_2=0} \oplus B \xrightarrow{ \id \oplus -\partial_3, \partial_2} B^2 \to B \to 0.$$
	Consider the subcomplex $$E_3: 0 \to B^{\partial_2,\partial_3=0} \xrightarrow{\mbox{\phantom{textaa}}} B \xrightarrow{\mbox{\phantom{textaaaaaa}}} B^2 \to B \to 0. \hspace{0.5cm}$$
	A small computation shows $H^0(E_3)=0.$ Furthermore  $H^i(E_3)=0$ for $i\geq2$ as they are unaffected by the change from $E_3'$ to $E_3.$
	Since the original sequence $E_3'$ is exact, it remains to show that the image of $f\colon B \xrightarrow{-\partial_3,\partial_2} B^2$ is the same as the image of $g\colon B^{\partial_2=0} \oplus B \xrightarrow{\id \oplus -\partial_3,\partial_2} B^2.$ One inclusion is obvious hence let us consider $(a,b)$ in the image of $g.$ Then there exist $(x,y)$ such that $(x -\partial_3y,\partial_2y)= (a,b).$ Since $x \in B^{\partial_2=0}$ we can find $x' \in B^{\partial_2=0}$ with $\partial_3x'=x$ by Lemma \ref{lem:poincare}. Define $c = x'+y \in B$ then $$f(c) = (\partial_3(x'+y),\partial_2(x'+y)) = (x+\partial_3y,\partial_2y) = (a,b).$$
	Inductively we obtain for $k=3,\dots,d-1$ in each step an exact sequence 
	$$0 \to B^{\partial_2,\dots,\partial_k=0} \to B \xrightarrow{\prod \partial_i}\prod_{i=2}^{k} B \to \dots.$$
	Take the $\partial_{k+1}$-cone to get 
	$$ 0 \to B^{\partial_2,\dots,\partial_k=0} \to B^{\partial_2,\dots,\partial_k=0} \oplus B \to  B \oplus \prod_{i=2}^{k} B  \dots$$
	and apply \textit(ii) of Lemma \ref{lem:poincare} to show that the subsequence
	$$ 0 \to B^{\partial_2,\dots,\partial_{k+1}=0} \to B  \xrightarrow{\prod \partial_i} \prod_{i=2}^{k+1} B \to \dots$$ is exact. 
\end{proof}
\begin{lem} \label{lem:dsigma}
	Let $\gamma \in \Gamma_L$ and let $a  = \chi_{LT}(\gamma).$ Then for every $\sigma\in \Sigma$ we have  $ \partial_\sigma \circ \gamma = [a]'^{\sigma}(y_\sigma) \gamma \circ \partial_\sigma$ and $\partial_\sigma \circ \varphi_L = [\pi_L]'^{\sigma}(y_\sigma) \varphi_L \circ \partial_\sigma.$
\end{lem}
\begin{proof}
	Recall that $\partial_{\sigma} = \frac{g_{\sigma}}{t_\sigma} \nabla_\sigma,$ where $g_{\sigma}$ is the formal derivative of $t_\sigma = \log_{LT}^ {\sigma}(y_\sigma)$ with respect to the variable $y_\sigma.$
	From $\gamma(y_\sigma) = [\chi_{LT}(\gamma)]^{\sigma}(y_\sigma)$ and $\varphi(y_\sigma) = [\pi_L]^\sigma(y_\sigma)$ we conclude $\gamma t_\sigma = \sigma(\chi_{LT}(\gamma)) t_\sigma$ and $\varphi_L(t_\sigma) = \sigma(\pi_L)t_\sigma.$
	From the chain rule we obtain $$\sigma(a)g_\sigma = \gamma(g_\sigma) [a]'^\sigma(y_\sigma)$$ and similary
	$$\sigma(\pi_L)g_\sigma = \varphi(g_\sigma) [\pi_L]'^\sigma(y_\sigma).$$
	Using that $\nabla_\sigma$ commutes with the $(\varphi,\Gamma_L)$-action and the above computations we obtain $$\gamma(\partial_\sigma(x))= \gamma(\frac{g_\sigma}{t_\sigma})\nabla_\sigma \gamma x = ([a]'^\sigma(y_\sigma))^{-1} \partial_\sigma(\gamma x)$$ and
	$$\varphi(\partial_\sigma(x))= \varphi(\frac{g_\sigma}{t_\sigma})\nabla_\sigma \varphi x = ([\pi_L]'^\sigma(y_\sigma))^{-1} \partial_\sigma(\varphi x).$$
\end{proof}
As a consequence of Lemma \ref{lem:dsigma} $\partial_\sigma$ does not commute with $\varphi_L.$ Instead we have that $\partial_\sigma$ defines a $(\varphi_L,\Gamma_L)$-equivariant map 
$$\partial_{\sigma}\colon B^I \to B^{I}d_\sigma; f \mapsto \partial_{\sigma}(f)d_\sigma,$$ where $d_\sigma$ is a formal symbol with $(\varphi,\Gamma)$-action given by $a\cdot d_\sigma = [a]'^\sigma(w_\sigma)d_\sigma.$\footnote{Here we mean that $(\varphi_L^n,\gamma) \in \varphi_L^{\NN_0} \times\Gamma$ acts as $a= \pi_L^n\chi_{LT}(\gamma)$.}

Since the Frobenius $\varphi_L\colon \tilde{\fB}^I \to \tilde{\fB}^{I^{1/q}}$ is not an endomorphism of $\tilde{\fB}^I$ we need to make the following rather technical definitions.
\begin{defn}
	Let $I  \in \{[r,s],[r,s)\}$ be a subinterval of $(0,1)$ and assume $s>r^{1/q}.$ Let $J:= [r^{1/q},s].$ 
	Let $$R^I \in \left\{\tilde{\fB}^I, (\tilde{\fB}^I)^{H_L} ,((\tilde{\fB}^I)^{H_L})^{pa},((\tilde{\fB}^I)^{H_L})^{L-pa}\right\}$$ (similarly for $R^J$). A $\varphi_L$-module $M^I$ over $R^I$ is a finite free $R^I$-module together with an isomorphism 
	$$R^J \otimes_{R^I,\varphi_L} M^I \cong M^J:=R^J\otimes_{R^I}M^I.$$
	
\end{defn}

\begin{defn}
	Let $M=M^I$ be a $\varphi_L$-module over $B=B^I$ endowed with a compatible and continuous action of $\Gamma_L$ such that $\nabla_\sigma(M) \subset t_{\sigma}M$ for all $\sigma \in \Sigma_0.$ We define the equivariant analytic resolution of $M$ as 
	$$ C^\bullet_{\Omega^{\Sigma_0}}(M): M \xrightarrow{\prod \partial_{\sigma_i}} \oplus_{i=2}^n M \otimes_B  Bd_{\sigma_i} \to \dots \xrightarrow{\sum (-1)^i \partial_{\sigma_i}} M\otimes_B Bd_{\sigma_2} \wedge \dots \wedge d_{\sigma_n},$$ which as a complex of vector spaces is given by $$K^\bullet(\partial_{\sigma_2},\dots,\partial_{\sigma_n},M)$$ endowed with an action of $(\varphi_L,\Gamma_L)$ induced by the action on $M$ and the symbols $d_{\sigma_i}.$ By construction each operator commutes with the action of $\gamma \in \Gamma_L$ and $\varphi_L$ and we may thus form
	$$C^\bullet_{\Omega^\Sigma}(M) := \operatorname{cone}[C_{\Omega^{\Sigma_0}}(M) \xrightarrow{\nabla_1} C_{\Omega^{\Sigma_0}}(M) ][-1]$$
	and
	$$C^\bullet_{\Omega^\Sigma,\varphi_L}(M) := \operatorname{cone} [C_{\Omega^{\Sigma}}(M^I) \xrightarrow{\varphi_L-1} C_{\Omega^{\Sigma}}(M^J)][-1].$$
\end{defn}

Formally we obtain the following analogues of the Frölicher spectral sequence in complex analysis:
\begin{rem}\label{rem:hodgederham}
	We have canonical convergent spectral sequences 
	$$E_{2}^{p,q} = H^p_{\nabla_1}(H^q_{\Omega^{\Sigma_0}}(M)) \implies H^{p+q}_{\Omega^{\Sigma}}(M)$$
	and
	$$E_{2}^{p,q} = H^p_{\varphi_L,\nabla_1}(H^q_{\Omega^{\Sigma_0}}(M)) \implies H^{p+q}_{\varphi_L,\Omega^{\Sigma}}(M).$$
\end{rem}
\begin{proof} 
	The first sequence is just the spectral sequence for the total complex of a double complex. 
	For the second spectral sequence an explicit computation shows that for $C^\bullet:= C^\bullet_{\Omega^{\Sigma_0}}(M)$ the Koszul complex\footnote{Formally, $K^\bullet(\varphi-1,C^\bullet)$ does not make sense, since $\varphi_L$ is not an endomorphism of $M^I$, we mean, of course, the mapping fibre of $\varphi-1 \colon C^\bullet_I \to C^\bullet_J,$ where $C_{I}^\bullet$ and $C_{J}^\bullet$ denote the respective resolutions for $M^I$ and $M^J.$ To keep the notation light we will suppress the intervals $I$ and $J$ in the notation. This poses no problem.} 
	$$K^\bullet(\nabla_1,\varphi-1,C^\bullet) (\simeq C^\bullet_{\Omega^\Sigma,\varphi}(M))$$ is isomorphic to the total complex of the double complex  
	\begin{equation} \label{eq:bigherr} C^\bullet \xrightarrow{\varphi-1,\nabla_1} C^\bullet \oplus C^\bullet \xrightarrow{\nabla_1 \oplus -(\varphi-1)} C^\bullet.	
	\end{equation}
	
	Indeed both total complexes are of the form $C^\bullet \oplus C^\bullet[-1] \oplus C^\bullet [-1] \oplus C^\bullet[-2]$ hence it remains to see that the differentials agree. 
	Let us denote by $D^\bullet$ the complex $K^\bullet(\nabla_1,C^\bullet)$ with differentials 
	$$\begin{pmatrix} d_C & 0 \\
		\nabla_1  & -d_C\end{pmatrix} .$$
	Hence $K^\bullet(\nabla_1,\varphi-1,C^\bullet)$ has differentials
	$$\begin{pmatrix} d_C & 0 & 0&0\\
		\nabla_1 & -d_C & 0&0\\
		\varphi_L-1  & 0 & -d_C & 0\\
		0& \varphi_L-1 & -\nabla_1 & d_C \end{pmatrix}, $$
	which agrees with the differential in the total complex of \eqref{eq:bigherr} (up to signs). 
	We obtain the second spectral sequence from the spectral sequence attached to \eqref{eq:bigherr}.
	
\end{proof}
\begin{thm} 
	Let $V\in \operatorname{Rep}_{L}(G_L)$ be $L$-analytic, let $I=[r,s]\subset (0,1]$ be a closed subinterval with $r$ large enough. Set $$N^I:= (\tilde{D}^I(V))^{la}:= ((\tilde{\fB}^I\otimes V)^{H_L})^{la}$$ and $M^I:= \operatorname{Sol}(N^I).$
	Then the natural map $M^I[0] \to C_{\Omega^{\Sigma_0}}^\bullet(N^I)$ is an equivalence. 
	In particular, the spectral sequences from Remark \ref{rem:hodgederham} collapse on the $E_2$-page and give $$H^i_{\Omega^{\Sigma}}(N^I) \cong H^i_{\nabla_1}(M^I)$$
	$$H^i_{\varphi_L,\Omega^{\Sigma}}(N^I) \cong H^i_{\varphi_L,\nabla_1}(M^I).$$
\end{thm}
\begin{proof}
	Since $V$ is $L$-analytic we can conclude from Theorem \ref{thm:monodromy}, that $$N = (\tilde{B}^I)^{la} \otimes_{(\tilde{B}^I_{L})^{L-la}} \operatorname{Sol}(M^I).$$ It follows from the Leibniz rule, that the operators $\partial_\sigma$ for $\sigma \in \Sigma_0$ act as $\partial_{\sigma} \otimes \id_{N}$ and hence the complex $C^\bullet_{\Omega^{\Sigma_0}}(N^I)$ is isomorphic to a power of the resolution of the trivial module. The claim follows from Lemma \ref{lem:dolbeaut}.
\end{proof}
This theorem provides us with a comparison isomorphism between $L$-Lie-algebra cohomology of $M^I$ and a twisted variant of $\Q_p$-Lie algebra cohomology (using the operators $\partial_i$ instead of $\nabla_i$ for $i>0$) of $N^I.$ We warn the reader that here $M^I$ is a module over the $L$-analytic vectors of $\tilde{\fB}_L^I$ rather than $\cR_L^I$ as considered in the preceding sections. By \cite[Theorem 4.4 (2)]{Berger2016} the former is explicitly given as $(\tilde{\fB}_L^I)^{la} = \bigcup_k \varphi_q^{-k}\mathbf{B}_L^{{I}^{1/q^k}},$ where $\mathbf{B}_L^?$ denotes the closure of (the image of) $L[T,T^{-1}]$ inside $\tilde{\fB}_L^?.$ The ring $\mathbf{B}_L^{{I}}$ is isomorphic to $\cR_L^{I'}$ for a suitable interval $I'$ (which can be precisely determined using the discussions after Lemma 3.4 \text{(loc.cit.)}).

\let\stdthebibliography\thebibliography
\let\stdendthebibliography\endthebibliography
\renewenvironment*{thebibliography}[1]{%
	\stdthebibliography{MSVW24}}
{\stdendthebibliography}

\bibliographystyle{abbrv}

\bibliography{Literatur}
\bigskip

\end{document}